\documentclass[11pt]{article}
\usepackage[margin=1in]{geometry}
\usepackage{amsmath,amssymb,amsthm,mathtools,bm}
\usepackage{enumitem}
\usepackage{hyperref}
\usepackage{booktabs}
\usepackage{graphicx}
\usepackage{float}
\usepackage{microtype}
\usepackage{xcolor}

\hypersetup{colorlinks=true,linkcolor=blue,citecolor=blue,urlcolor=blue}

\newtheorem{theorem}{Theorem}[section]
\newtheorem{proposition}[theorem]{Proposition}
\newtheorem{lemma}[theorem]{Lemma}

\newtheorem{definition}[theorem]{Definition}

\theoremstyle{remark}
\newtheorem{remark}[theorem]{Remark}
\newtheorem{example}[theorem]{Example}

\title{A Circular Chatterjee's Correlation Coefficient}
\author{Sourav Majumdar\\
Department of Management Sciences, Indian Institute of Technology Kanpur\\
\texttt{souravm@iitk.ac.in}}
\date{}

\begin{document}
\maketitle

\begin{abstract}
Chatterjee's rank correlation is a directed measure of association designed to detect whether one variable can be predicted as a function of another. While the original coefficient is naturally defined for real-valued data, circular data poses additional difficulty. Applying the usual construction requires cutting each circle at an arbitrary point and treating it as a line. Different choices of cut points can lead to different finite-sample values, even though the underlying circular relationship is unchanged. This paper proposes a circular version of Chatterjee's coefficient that removes this arbitrary choice. The population construction averages over response cuts in circular rank space, and the finite-sample construction averages over sample cut gaps and reduces to a simple statistic based only on cyclic ranks. The resulting coefficient is intrinsic to the circular ordering of the data, remains directed, and retains the key interpretation of Chatterjee's original coefficient. Under non-atomic circular marginals, it is zero exactly under independence and one exactly when the circular response is a measurable function of the circular predictor. We prove consistency and derive its distribution-free null behavior under independence. Simulations show that the proposed coefficient is especially useful for detecting multi-winding circular relationships, such as cases where the response goes around the circle twice or four times as the predictor goes around once, where standard circular correlations can be nearly blind.
\end{abstract}
\section{Introduction}

Circular observations arise whenever measurements are directions, phases, times of day, orientations, or other quantities defined modulo a full turn.  In such problems there is no distinguished origin and no natural place to cut the circle.  Cutting the circle at zero and treating angles as ordinary real numbers is often convenient, but it creates an artificial boundary by placing some nearby directions at opposite ends of the interval.  This is especially problematic for rank methods whose value depends on the adjacency structure of the ordered observations.  The aim of this paper is to construct a Chatterjee-type coefficient for circular variables that uses only the cyclic order of the data and does not depend on an arbitrary cut point.

Chatterjee's coefficient of correlation \cite{chatterjee2021} is a rank-based directed measure of association.  For real-valued data without ties, one sorts the observations by $X$, records the ranks $r_i$ of the concomitant $Y$ values, and computes
\[
  \xi_n(X,Y)=1-\frac{3}{n^2-1}\sum_{i=1}^{n-1}|r_{i+1}-r_i|.
\]
Its population limit $\xi(X,Y)$ lies in $[0,1]$, equals zero if and only if $X$ and $Y$ are independent, and equals one if and only if $Y=f(X)$ almost surely for some measurable function $f$.  The coefficient is intentionally directed, in the sense that it measures the extent to which the response is predictable as a function of the predictor.  A symmetric version can be obtained by taking the maximum of the two directions.  These two extreme cases are central to the interpretation of the coefficient. The value zero characterizes independence, while the value one characterizes measurable functional dependence.  We refer to these as the zero-one characterizations of the coefficient.  Together with the simple null theory, they are the main reasons for the coefficient's appeal \cite{chatterjee2021,chatterjee2023survey}.  The same population target had appeared earlier in regression-dependence form in work of Dette, Siburg and Stoimenov \cite{dettesiburgstoimenov2013}, but Chatterjee's statistic gave a particularly simple rank estimator and a direct asymptotic theory.

The literature following Chatterjee's paper has developed several aspects of this idea.  Azadkia and Chatterjee \cite{azadkiachatterjee2021} introduced a related coefficient of conditional dependence and used it for nonparametric variable selection.  Deb, Ghosal and Sen \cite{debghosalsen2020} proposed a broad graph-and-kernel framework for measuring association on general topological spaces, while Han and Huang \cite{hanhuang2022} studied adaptation of Azadkia-Chatterjee type methods to manifold-supported data.  The power behavior of Chatterjee's statistic has also been investigated.  Shi, Drton and Han \cite{shidrtonhan2022} showed that the independence test based on the original coefficient can be rate-suboptimal against standard smooth local alternatives, and Lin and Han \cite{linhan2023} proposed a boosted version using more than one nearby rank.  These developments show that Chatterjee's coefficient is best viewed not merely as a test of independence, but as a directed measure of functional predictability with a distinctive local rank-increment structure.

Circular and directional statistics have their own long literature on association.  Early work includes Johnson and Wehrly's measures and models for angular association \cite{johnsonwehrly1977}, Jupp and Mardia's general correlation coefficient for directional data \cite{juppmardia1980}, and Fisher and Lee's circular-circular correlation coefficient based on pairwise angular differences \cite{fisherlee1983}.  The centered-sine coefficient of Jammalamadaka and Sarma \cite{jammalamadakasarma1988}, discussed in the monographs of Jammalamadaka and SenGupta \cite{jammalamadakasengupta2001} and Mardia and Jupp \cite{mardiajupp2000}, is another standard measure.  Related work has studied tests of independence and association for bivariate circular data, including Rothman's tests on the torus \cite{rothman1971}, Puri and Rao's tests based on circular association \cite{purirao1977}, and the weighted-degenerate $U$-statistic approach of Shieh, Johnson and Frees \cite{shiehjohnsonfrees1994}.  More recent contributions include Zhan, Ma, Liu and Shimizu's work on circular correlation for data on the torus \cite{zhanma2019}, Chakraborty and Wong's study of Jammalamadaka-Sarma and Fisher-Lee correlations for bivariate von Mises models \cite{chakrabortywong2023}, and the recent graph-based tests for randomness in circular data proposed by Gehlot and Laha \cite{gehlotlaha2025}.  These works address important aspects of circular association, circular independence, or randomness, but they do not provide a directed Chatterjee-type coefficient whose zero-one characterization corresponds to independence and measurable functional dependence.

A formal route from Chatterjee's coefficient to circular variables is available through the standard-Borel extension described in Chatterjee's survey \cite{chatterjee2023survey}.  Let $\mathcal X$ and $\mathcal Y$ be standard Borel spaces.  Choose Borel isomorphisms
\[
  \phi:\mathcal X\to \phi(\mathcal X)\subset\mathbb R,
  \qquad
  \psi:\mathcal Y\to \psi(\mathcal Y)\subset\mathbb R,
\]
where the ranges are Borel subsets of the real line.  Given data $(X_1,Y_1),\ldots,(X_n,Y_n)$ with values in $\mathcal X\times\mathcal Y$, define
\[
  \xi_n^{\phi,\psi}(X,Y)
  :=\xi_n\{\phi(X),\psi(Y)\},
\]
that is, apply the ordinary real-line statistic to the encoded observations $(\phi(X_i),\psi(Y_i))$.  Since a Borel isomorphism preserves the measurable structure, independence of $X$ and $Y$ is equivalent to independence of $\phi(X)$ and $\psi(Y)$, and the assertion that $Y$ is a measurable function of $X$ is equivalent to the assertion that $\psi(Y)$ is a measurable function of $\phi(X)$.  Thus the usual consistency and zero-one interpretation transfer to this encoded version.

This proposal applies formally to circular variables, because the circle is a compact Polish space and hence a standard Borel space.  The most natural way to implement it on $S^1$ is to choose a cut point and use it to identify the circle with an interval.  Writing the circle as $\mathbb T=\mathbb R/\mathbb Z$, a cut point $a\in\mathbb T$ gives
\[
  \phi_a(x)=[x-a]_1,
\]
where $[\cdot]_1$ denotes reduction modulo one into $[0,1)$.  Similarly, a response cut point $b$ gives $\psi_b(y)=[y-b]_1$.  For each pair $(a,b)$, the statistic $\xi_n^{a,b}$ obtained by applying Chatterjee's ordinary coefficient to $(\phi_a(X_i),\psi_b(Y_i))$ is therefore a valid instance of the standard-Borel construction.

The difficulty is that this construction is an existence result rather than a canonical circular coefficient.  First, the value depends on the chosen isomorphisms.  Even within the geometrically natural family $\{\phi_a,\psi_b\}$, changing $a$ or $b$ may change the finite-sample statistic, although those cut points are not features of the data-generating problem.  Second, a linearization turns a cyclic order into a linear order and creates an artificial boundary by placing some angles that are adjacent on the circle at opposite ends of the interval.  Since Chatterjee's statistic is based on adjacent rank increments, this boundary can affect the increments that the statistic sees.  Third, the standard-Borel theorem allows many Borel isomorphisms that are far less natural than cut-point linearizations; such maps can ignore rotations, reversals, arc distance, and cyclic order.  Thus the standard-Borel construction guarantees that some real-line encoding can be used, but it does not specify a statistic determined only by the circular sample and the circular geometry.

This paper constructs such a statistic.  At the population level, we average Chatterjee's ordinary coefficient over response cuts in circular rank space; at the finite-sample level, the corresponding average is over predictor and response sample cut gaps.  The averaging is not merely a conceptual device, since in finite samples it collapses to a closed-form cyclic-rank statistic.  If $X$ determines a cyclic order and $Y$ determines cyclic ranks, only the cyclic rank increments of $Y$ along adjacent $X$-edges are needed.  The resulting coefficient is invariant under rotations and reversals of the circle, and more generally under orientation-preserving or orientation-reversing circular reparameterizations at the level of cyclic order.  It keeps Chatterjee's directed interpretation--nearby predictor ranks should have nearby response ranks when the response is predictable from the predictor--but replaces the arbitrary Borel linearization by an intrinsic cyclic-rank formula that is computable in $O(n\log n)$ time.

Section~2 defines the population coefficient $\xi^\circ(X\to Y)$ for circular variables and gives three equivalent forms, namely a conditional circular-dispersion representation, a cut-average representation, and a Fourier representation.  The Fourier form is then used to prove the zero-one characterizations. For non-atomic circular marginals, the coefficient is zero exactly under independence and one exactly under measurable functional dependence.  Section~3 derives the finite-sample cyclic-rank statistic and proves that it is the exact average of Chatterjee's ordinary finite-sample coefficient over all predictor and response sample cut gaps.  Section~4 proves strong consistency.  Section~5 works out the exact continuous-null mean and variance and establishes the asymptotic null law using a combinatorial central limit theorem for random cycles.  Section~6 gives explicit population examples for circular additive-noise models, and Section~7 summarizes the algorithm and the directed and symmetric versions.  Section~8 presents simulations, including comparisons with standard circular correlations, cut-point sensitivity of the Borel construction, and null calibration of the proposed test.

The construction is different from general metric-space or kernel extensions of Chatterjee-type ideas, such as those of Deb, Ghosal and Sen \cite{debghosalsen2020}.  Those frameworks cover far more spaces.  The point here is not maximal generality, but the special cyclic-order structure of $S^1$, which yields a closed-form rank statistic and an exact cyclic-permutation null theory.

\subsection{Notation}

Denote $\mathbb{T}=\mathbb{R}/\mathbb Z$ for the unit circle, with addition modulo one. The same formulas apply to angles in $[0,2\pi)$ after division by $2\pi$. For $u,v\in\mathbb{T}$, define the counterclockwise arc length
\[
  [v-u]_1 \in [0,1)
\]
to be the unique representative of $v-u$ modulo one. Let
\[
  h(t)=t(1-t),\qquad 0\leq t\leq 1.
\]
The circular discrepancy between two circular ranks $u,v$ is
\[
  H(u,v):=h([v-u]_1)=[v-u]_1(1-[v-u]_1).
\]
This discrepancy is symmetric, since $h(t)=h(1-t)$, although $[v-u]_1$ is oriented.

Let $X,Y$ be $\mathbb{T}$-valued random variables with non-atomic marginal distributions $\mu_X,\mu_Y$. Choose arbitrary origins $o_X,o_Y\in\mathbb{T}$ and define the circular distributional ranks
\[
  S=F_X(X):=\mu_X([o_X,X)),\qquad
  U=F_Y(Y):=\mu_Y([o_Y,Y)),
\]
where $[o,z)$ denotes the counterclockwise arc from $o$ to $z$. Since the marginals are non-atomic, $S$ and $U$ are uniform on $\mathbb{T}$. Changing the origins adds constants to $S$ and $U$ modulo one, so all differences and all coefficients below are unaffected.

For $m\in\mathbb Z$, define the conditional Fourier coefficient
\[
  g_m(s):=\mathbb{E}(e^{2\pi i m U}\mid S=s),\qquad s\in\mathbb{T}.
\]
We use a regular conditional distribution of $U$ given $S=s$; it exists because $\mathbb{T}$ is a standard Borel space.

\section{The population coefficient}

\begin{definition}[Circular Chatterjee coefficient]
Let $K(s,\cdot)$ be a regular conditional distribution of $U$ given $S=s$.  Given $S=s$, let $U_1$ and $U_2$ be conditionally independent with common law $K(s,\cdot)$.  Define
\begin{equation}\label{eq:pop-disp}
  \xi^\circ(X\to Y)
  :=1-6\,\mathbb{E}\,H(U_1,U_2).
\end{equation}
Equivalently,
\[
  \xi^\circ(X\to Y)=1-6\,\mathbb{E}\Big([U_2-U_1]_1(1-[U_2-U_1]_1)\Big).
\]
\end{definition}

The normalization is chosen so that the coefficient is zero under independence. If $U_1,U_2$ are independent uniform circular ranks, then $[U_2-U_1]_1$ is uniform on $[0,1]$, and $\mathbb{E} h([U_2-U_1]_1)=1/6$.

We first record two measure-theoretic facts that will be used in the proof of the zero-one characterizations.  They are the circular analogues of the determining-class and conditional-law steps in Chatterjee's proof of the corresponding zero and one characterizations for the ordinary coefficient.

\begin{lemma}\label{lem:rank-isomorphism}
Let $Z$ be a $\mathbb{T}$-valued random variable with non-atomic law $\mu$.  Fix an origin $o$ and define
\[
  R_\mu(z)=\mu([o,z)),\qquad z\in\mathbb{T}.
\]
Then $R_\mu(Z)$ is uniform on $\mathbb{T}$.  Moreover, there is a Borel map $Q_\mu:\mathbb{T}\to\mathbb{T}$ such that
\[
  Q_\mu(R_\mu(Z))=Z\qquad\text{a.s.}
\]
Consequently, for non-atomic circular marginals,
\begin{align*}
  X\perp Y
  &\quad\Longleftrightarrow\quad R_{\mu_X}(X)\perp R_{\mu_Y}(Y),\\
  Y=f(X)\text{ a.s. for some Borel }f
  &\quad\Longleftrightarrow\quad
  R_{\mu_Y}(Y)=\phi(R_{\mu_X}(X))\text{ a.s. for some Borel }\phi.
\end{align*}
\end{lemma}

\begin{proof}
After cutting the circle at $o$, identify it with $[0,1)$.  Let
\[
  F(t)=\mu([0,t]),\qquad 0\leq t<1,
\]
with the convention $F(1)=1$.  Since $\mu$ has no atoms, $F$ is continuous.  The probability integral transform gives $F(Z)\sim \mathrm{Unif}[0,1]$; with $F(t)=\mu([0,t])=\mu([0,t))$ at $\mu$-a.e. $t$, this is the same as the circular rank $R_\mu(Z)$.

Define the generalized inverse
\[
  Q_\mu(u)=\inf\{t\in[0,1]:F(t)\geq u\},
\]
and map the resulting point of $[0,1)$ back to $\mathbb{T}$.  We claim that $Q_\mu(F(Z))=Z$ a.s.  Indeed, for a continuous nondecreasing $F$, the set of $t$ for which $Q_\mu(F(t))\neq t$ is contained in the union of the open intervals on which $F$ is constant, together with the endpoints of those intervals.  The open intervals are disjoint, hence countable, and each has $\mu$-measure zero; the endpoints are countable and also have $\mu$-measure zero because $\mu$ is non-atomic.  Therefore the exceptional set has $\mu$-measure zero, proving the inverse identity.

The two consequences follow immediately.  If $S=R_{\mu_X}(X)$ and $U=R_{\mu_Y}(Y)$ are independent, then $X=Q_{\mu_X}(S)$ and $Y=Q_{\mu_Y}(U)$ a.s., so $X$ and $Y$ are independent.  The converse is trivial because $S$ and $U$ are measurable functions of $X$ and $Y$.  Similarly, if $Y=f(X)$ a.s., then
\[
  U=R_{\mu_Y}\{f(Q_{\mu_X}(S))\}\qquad\text{a.s.},
\]
so $U$ is a Borel function of $S$.  Conversely, if $U=\phi(S)$ a.s., then
\[
  Y=Q_{\mu_Y}\{\phi(R_{\mu_X}(X))\}\qquad\text{a.s.},
\]
so $Y$ is a Borel function of $X$.
\end{proof}

\begin{lemma}\label{lem:conditional-law-characterizations}
Let $K(s,\cdot)$ be a probability kernel from $\mathbb{T}$ to $\mathbb{T}$, and write
\[
  \widehat K_s(m)=\int_\mathbb{T} e^{2\pi i m u}\,K(s,du),\qquad m\in\mathbb Z.
\]
Then the following statements hold.
\begin{enumerate}[label=(\alph*),leftmargin=2em]
\item If $\widehat K_s(m)=0$ for every nonzero $m\in\mathbb Z$, then $K(s,\cdot)$ is Haar measure on $\mathbb{T}$.
\item If
\[
  \int_\mathbb{T}\int_\mathbb{T} H(u,v)\,K(s,du)K(s,dv)=0,
\]
then $K(s,\cdot)$ is a point mass.  Conversely, every point mass makes this integral zero.
\item If $K(s,\cdot)$ is a point mass for a.e. $s$, then there is a Borel map $\phi:\mathbb{T}\to\mathbb{T}$ such that $K(s,\cdot)=\delta_{\phi(s)}$ for a.e. $s$.
\end{enumerate}
\end{lemma}

\begin{proof}
For (a), Haar measure has Fourier coefficients equal to zero at every nonzero integer and equal to one at zero.  Thus $K(s,\cdot)$ and Haar measure agree on every trigonometric polynomial.  Trigonometric polynomials are uniformly dense in $C(\mathbb{T})$ by Stone-Weierstrass, so the two measures agree on every continuous function.  By the Riesz representation theorem, the two probability measures are equal.

For (b), note that $H$ is nonnegative and $H(u,v)=0$ if and only if $u=v$ in $\mathbb{T}$.  Hence a zero integral implies
\[
  K(s,\cdot)\otimes K(s,\cdot)\{(u,v):u\neq v\}=0.
\]
If $K(s,\cdot)$ were not a point mass, there would be a Borel set $A$ with $0<K(s,A)<1$.  Then
\[
  K(s,\cdot)\otimes K(s,\cdot)(A\times A^c)=K(s,A)\{1-K(s,A)\}>0,
\]
contradicting concentration on the diagonal.  Therefore $K(s,\cdot)$ is a point mass.  The converse is immediate.

For (c), define
\[
  \gamma(s)=\int_\mathbb{T} e^{2\pi i u}\,K(s,du).
\]
The map $s\mapsto\gamma(s)$ is measurable by the definition of a probability kernel.  On the full-measure set where $K(s,\cdot)$ is a point mass, $|\gamma(s)|=1$ and $\gamma(s)$ is exactly the point $e^{2\pi i u}$ corresponding to that atom.  Let $\arg_\mathbb{T}$ be any Borel inverse of $u\mapsto e^{2\pi i u}$ from the unit circle in $\mathbb C$ to $\mathbb{T}$, and define $\phi(s)=\arg_\mathbb{T}\gamma(s)$ on this set, with an arbitrary value elsewhere.  Then $\phi$ is Borel and $K(s,\cdot)=\delta_{\phi(s)}$ for a.e. $s$.
\end{proof}

The function $h(t)=t(1-t)$ has the cosine expansion
\begin{equation}\label{eq:h-fourier}
  h(t)=\frac16-\frac{1}{\pi^2}\sum_{m=1}^{\infty}\frac{\cos(2\pi m t)}{m^2},\qquad 0\leq t\leq 1.
\end{equation}
The series is absolutely and uniformly convergent.

\begin{theorem}[Fourier representation and zero-one characterization]\label{thm:fourier-endpoints}
For non-atomic circular marginals,
\begin{equation}\label{eq:fourier-pop}
  \xi^\circ(X\to Y)
  =\frac{6}{\pi^2}\sum_{m=1}^{\infty}\frac{\mathbb{E}|g_m(S)|^2}{m^2}.
\end{equation}
Consequently
\[
  0\leq \xi^\circ(X\to Y)\leq 1.
\]
Moreover,
\begin{align*}
  \xi^\circ(X\to Y)=0
  &\quad\Longleftrightarrow\quad X\text{ and }Y\text{ are independent},\\
  \xi^\circ(X\to Y)=1
  &\quad\Longleftrightarrow\quad Y=f(X)\text{ a.s. for some measurable }f:\mathbb{T}\to\mathbb{T}.
\end{align*}
\end{theorem}

\begin{proof}
Let $K(s,\cdot)$ be a regular conditional distribution of $U$ given $S=s$.  Thus
\[
  g_m(s)=\int_\mathbb{T} e^{2\pi i m u}\,K(s,du)
\]
for a.e. $s$.  Conditional on $S=s$, let $U_1,U_2$ be independent with common law $K(s,\cdot)$.  Then, for each $m\geq 1$,
\[
\begin{aligned}
  \mathbb{E}\{e^{2\pi i m(U_2-U_1)}\mid S=s\}
  &=\left(\int e^{2\pi i m v}\,K(s,dv)\right)
    \left(\int e^{-2\pi i m u}\,K(s,du)\right)  \\
  &=|g_m(s)|^2.
\end{aligned}
\]
Substituting the Fourier expansion \eqref{eq:h-fourier} into $H(U_1,U_2)=h([U_2-U_1]_1)$ gives, for a.e. $s$,
\[
  \mathbb{E}\{H(U_1,U_2)\mid S=s\}
  =\frac16-\frac1{\pi^2}\sum_{m=1}^\infty\frac{|g_m(s)|^2}{m^2}.
\]
The interchange of the conditional expectation and the series is justified by absolute summability.  Taking expectations in $s$ and using the definition \eqref{eq:pop-disp} gives \eqref{eq:fourier-pop}.

The upper bound $\xi^\circ\leq1$ follows directly from \eqref{eq:pop-disp}, since $H\geq0$.  The lower bound follows from the Fourier formula \eqref{eq:fourier-pop}, since every summand is nonnegative.

We first prove the characterization of the value zero.  If $X$ and $Y$ are independent, then $S$ and $U$ are independent by Lemma \ref{lem:rank-isomorphism}.  Since $U$ is uniform on $\mathbb{T}$,
\[
  g_m(S)=\mathbb{E}(e^{2\pi i mU}\mid S)=\mathbb{E} e^{2\pi i mU}=0,
  \qquad m\geq1,
\]
so \eqref{eq:fourier-pop} gives $\xi^\circ=0$.

Conversely, suppose $\xi^\circ=0$.  Since the terms in \eqref{eq:fourier-pop} are nonnegative,
\[
  \mathbb{E}|g_m(S)|^2=0\qquad\text{for every }m\geq1.
\]
Because the set of positive integers is countable, there is a single Borel set $B\subset\mathbb{T}$ with Lebesgue measure one such that $g_m(s)=0$ for every $m\geq1$ and every $s\in B$.  The negative Fourier coefficients also vanish, because they are complex conjugates of the positive ones.  Hence, by Lemma \ref{lem:conditional-law-characterizations}(a),
\[
  K(s,\cdot)=\lambda(\cdot)\qquad\text{for every }s\in B,
\]
where $\lambda$ is Haar measure on $\mathbb{T}$.  Therefore, for Borel sets $A,C\subset\mathbb{T}$,
\[
\begin{aligned}
  \mathbb{P}(S\in A, U\in C)
  &=\int_A K(s,C)\,\lambda(ds) \\
  &=\int_A \lambda(C)\,\lambda(ds)
   =\lambda(A)\lambda(C).
\end{aligned}
\]
Thus $S$ and $U$ are independent.  Lemma \ref{lem:rank-isomorphism} then gives independence of $X$ and $Y$.

We now prove the characterization of the value one.  If $Y=f(X)$ a.s., then Lemma \ref{lem:rank-isomorphism} gives a Borel $\phi$ such that $U=\phi(S)$ a.s.  Therefore $K(s,\cdot)=\delta_{\phi(s)}$ for a.e. $s$.  Two conditionally independent draws from a point mass are equal, so $H(U_1,U_2)=0$ a.s. and \eqref{eq:pop-disp} gives $\xi^\circ=1$.

Conversely, suppose $\xi^\circ=1$.  By \eqref{eq:pop-disp},
\[
  \mathbb{E} H(U_1,U_2)=0.
\]
Since $H\geq0$, this implies that for a.e. $s$,
\[
  \int_\mathbb{T}\int_\mathbb{T} H(u,v)\,K(s,du)K(s,dv)=0.
\]
By Lemma \ref{lem:conditional-law-characterizations}(b), $K(s,\cdot)$ is a point mass for a.e. $s$.  By Lemma \ref{lem:conditional-law-characterizations}(c), there is a Borel $\phi:\mathbb{T}\to\mathbb{T}$ such that $K(s,\cdot)=\delta_{\phi(s)}$ for a.e. $s$.  Hence
\[
  \mathbb{P}(U=\phi(S))=\int K(s,\{\phi(s)\})\,\lambda(ds)=1.
\]
Applying Lemma \ref{lem:rank-isomorphism} once more, $Y$ is a Borel measurable function of $X$ a.s.
\end{proof}

\begin{remark}
For the ordinary coefficient, Chatterjee proves the zero characterization using the conditional survival functions $G_X(t)=P(Y\geq t\mid X)$ and $G(t)=P(Y\geq t)$.  Equality $G_X(t)=G(t)$ is first obtained for a full-$\mu$ set of thresholds and is then extended to all thresholds by the atom, right-continuity, and support structure of the law of $Y$; half-lines then form a determining class.  He proves the one characterization by showing that the conditional survival function is $0$ or $1$ for $\mu$-almost every threshold.  Equivalently, the conditional law is trapped in an interval of zero effective $\mu$-width, which the support argument forces to be a single point, so the conditional law is Dirac.  The proof above follows the same logic in the circular setting, but replaces half-line indicators by the Fourier characters of $\mathbb{T}$.  The nonzero Fourier characters form a countable determining class for probability measures on the circle, and the condition $\xi^\circ=1$ is exactly the statement that $H(U_1,U_2)=0$ a.s. for two conditionally independent draws from $K(S,\cdot)$.  Since $H(u,v)=0$ if and only if $u=v$ on the circle, this forces $K(S,\cdot)$ to be a Dirac law a.s.
\end{remark}

\subsection{Cut-average interpretation}

For a response cut point $b\in\mathbb{T}$, let
\[
  U^{(b)}=[U-b]_1\in[0,1).
\]
If $\xi_b(X,Y)$ denotes Chatterjee's population coefficient applied to the real-valued response $U^{(b)}$ and to any fixed linear cut of the predictor, then
\begin{equation}\label{eq:pop-cut-average}
  \xi^\circ(X\to Y)=\int_0^1 \xi_b(X,Y)\,db.
\end{equation}
This is a response-cut average in circular rank space.  The predictor cut is immaterial at the population level, because any fixed linear cut of the predictor rank generates the same $\sigma$-field as $S$, whereas the response cut changes the ordinary linear rank distance.

Indeed, for two fixed circular ranks $u,v$ with $q=[v-u]_1$,
\begin{equation}\label{eq:cut-distance-identity}
  \int_0^1 \left|[u-b]_1-[v-b]_1\right|\,db=2q(1-q)=2H(u,v).
\end{equation}
The ordinary Chatterjee population coefficient for a continuous linear response can be written as $1-3\mathbb{E}|U^{(b)}_1-U^{(b)}_2|$, where $U^{(b)}_1,U^{(b)}_2$ are conditionally independent given the predictor. Averaging over $b$ and applying \eqref{eq:cut-distance-identity} yields \eqref{eq:pop-cut-average}.

\section{The finite sample cyclic-rank statistic}

Assume first that there are no ties among the observed $X_i$'s or $Y_i$'s. Let
\[
  (X_1,Y_1),\ldots,(X_n,Y_n)
\]
be iid copies of $(X,Y)$. Order the indices cyclically around the predictor circle as
\[
  i_1,i_2,\ldots,i_n,
\]
with $i_{n+1}=i_1$. Let $\rho_j\in\{0,1,\ldots,n-1\}$ be the cyclic rank of $Y_j$ around the response circle, computed from any arbitrary origin. Changing the origin adds the same constant modulo $n$ to every $\rho_j$ and hence does not change the statistic. Define the clockwise cyclic rank increment along the $k$th predictor edge as
\begin{equation}\label{eq:dk}
  d_k=(\rho_{i_{k+1}}-\rho_{i_k})\bmod n\in\{1,2,\ldots,n-1\}.
\end{equation}

\begin{definition}[Finite sample circular coefficient]\label{def:raw-stat}
For $n\geq 2$, define
\begin{equation}\label{eq:raw-stat}
  \xi_n^\circ(X\to Y)
  :=1-\frac{6}{n^2(n+1)}\sum_{k=1}^n d_k(n-d_k).
\end{equation}
Equivalently,
\[
  \xi_n^\circ(X\to Y)=1-\frac{6}{n+1}\sum_{k=1}^n h(d_k/n).
\]
\end{definition}

The raw statistic is a literal equal-weight average over sample cut gaps of the ordinary coefficient. Its finite-sample maximum is less than one, as in Chatterjee's original statistic. It is useful to introduce the finite-sample corrected version
\begin{equation}\label{eq:corrected-stat}
  \xi_{n,*}^\circ(X\to Y)
  :=\frac{\xi_n^\circ(X\to Y)}{a_n},
  \qquad
  a_n:=1-\frac{6(n-1)}{n(n+1)}=\frac{(n-2)(n-3)}{n(n+1)},
\end{equation}
for $n\geq 4$. This rescaling is asymptotically negligible, since $a_n\to 1$, but it makes the maximum equal to one when the two cyclic rank orders agree or are exactly reversed.

\begin{proposition}[Exact sample cut average]\label{prop:sample-cut-average}
Let $\xi_n^{a,b}$ be Chatterjee's ordinary finite-sample coefficient obtained by cutting the predictor circle at the $a$-th sample gap and the response circle at the $b$-th sample gap, then applying the usual linear ranks. Then
\begin{equation}\label{eq:sample-cut-average}
  \xi_n^\circ(X\to Y)=\frac1{n^2}\sum_{a=1}^n\sum_{b=1}^n \xi_n^{a,b}.
\end{equation}
\end{proposition}

\begin{proof}
Fix a cyclic predictor edge $e_k=(i_k,i_{k+1})$ and write $d=d_k$. For a fixed response sample gap, the ordinary absolute linear rank difference across this edge is either $d$ or $n-d$. As the response cut runs through the $n$ sample gaps, the value is $d$ for $n-d$ sample gaps and $n-d$ for $d$ sample gaps. Therefore its average over response sample gaps is
\[
  \frac{(n-d)d+d(n-d)}{n}=\frac{2d(n-d)}{n}.
\]
For a fixed predictor sample gap, the ordinary Chatterjee sum contains all cyclic predictor edges except the one crossing the cut. Averaging over the $n$ predictor sample gaps includes each cyclic edge exactly $n-1$ times. Thus the double average of the ordinary rank-increment sum is
\[
  \frac{n-1}{n}\sum_{k=1}^n \frac{2d_k(n-d_k)}{n}.
\]
Multiplying by the ordinary no-tie Chatterjee normalizing factor $3/(n^2-1)$ gives
\[
  \frac{3}{n^2-1}\cdot \frac{n-1}{n}\cdot \frac{2}{n}
  \sum_{k=1}^n d_k(n-d_k)
  =\frac{6}{n^2(n+1)}\sum_{k=1}^n d_k(n-d_k),
\]
which proves \eqref{eq:sample-cut-average}.
\end{proof}

\begin{proposition}[Invariance]\label{prop:invariance}
The statistic $\xi_n^\circ$ is invariant under rotations, reflections, and arbitrary orientation-preserving or orientation-reversing homeomorphisms applied separately to the predictor and response circles. Also,
\[
  -\frac{6}{n^2(n+1)}\max_{\pi}\sum_{k=1}^n d_k(n-d_k)+1
  \leq \xi_n^\circ\leq a_n,
\]
where $a_n$ is given in \eqref{eq:corrected-stat}. The upper bound is attained exactly when the cyclic rank orders agree or are reversed.
\end{proposition}

\begin{proof}
A circle homeomorphism either preserves or reverses cyclic order. If it preserves cyclic order, the increments $d_k$ are unchanged up to a cyclic relabeling. If it reverses the response orientation, each $d_k$ is replaced by $n-d_k$; the product $d_k(n-d_k)$ is unchanged. If it reverses the predictor orientation, the same unoriented cyclic edges are traversed in the opposite direction, again replacing increments by $n-d_k$. Hence the statistic is invariant.

Since $d(n-d)\geq n-1$ for $d\in\{1,\ldots,n-1\}$, the sum in \eqref{eq:raw-stat} is at least $n(n-1)$, giving the upper bound $a_n$. Equality requires every $d_k$ to be $1$ or every $d_k$ to be $n-1$, which is exactly agreement or reversal of the two cyclic rank orders. The lower bound is immediate from maximizing the nonnegative sum over cyclic permutations.
\end{proof}

\section{Consistency}\label{sec:consistency}

This section gives the technical proof of consistency.  It is the circular analogue of the consistency theorem for Chatterjee's original coefficient, but the cyclic setting lets us separate the proof into two elementary ingredients, namely uniform convergence of empirical arc probabilities and a successor-based law of large numbers for bounded Fourier functions.

Let
\[
  S_i=F_X(X_i),\qquad U_i=F_Y(Y_i),
\]
where the circular distributional ranks are defined in Section 2.  Then $S_i$ and $U_i$ are marginally uniform on $\mathbb{T}$.  Because the marginals are non-atomic, ties occur with probability zero.  Let $N(i)$ be the clockwise successor of $S_i$ among $S_1,\ldots,S_n$.  Let
\[
  A_i=[U_{N(i)}-U_i]_1.
\]
Thus $A_i$ is the population circular $Y$-rank increment along the edge from $S_i$ to its empirical clockwise successor.

\begin{lemma}[Uniform empirical control of circular arcs]\label{lem:vc-arcs}
Let
\[
  \lambda_n(I)=\frac1n\sum_{j=1}^n\mathbf{1}\{U_j\in I\}
\]
for circular arcs $I\subset\mathbb{T}$, and let $\lambda$ be Haar measure on $\mathbb{T}$.  Then
\[
  \sup_{I\in \text{ circular arc}} |\lambda_n(I)-\lambda(I)|\longrightarrow 0
\]
almost surely.
\end{lemma}

\begin{proof}
A circular arc is either an interval in $[0,1)$ or the complement of such an interval, after a fixed identification of $\mathbb{T}$ with $[0,1)$.  Hence the asserted supremum is bounded by twice the usual Kolmogorov-Smirnov supremum for intervals on the line.  Since $U_1,U_2,\ldots$ are iid uniform random variables, the Glivenko-Cantelli theorem gives the result. 
\end{proof}

\begin{lemma}[Sample cyclic ranks approximate circular ranks]\label{lem:rank-approx}
Let $d_i$ be the sample cyclic rank increment in \eqref{eq:dk} along the edge from $i$ to $N(i)$.  Then, almost surely,
\begin{equation}\label{eq:rank-approx}
  \max_{1\leq i\leq n}\left|\frac{d_i}{n}-A_i\right|\longrightarrow 0.
\end{equation}
\end{lemma}

\begin{proof}
Let $I_i$ be the clockwise arc from $U_i$ to $U_{N(i)}$, including the endpoint $U_{N(i)}$ and excluding $U_i$.  Since there are no ties, $d_i$ is exactly the number of sample $U$-values in this arc.  Therefore
\[
  \frac{d_i}{n}=\lambda_n(I_i),\qquad A_i=\lambda(I_i).
\]
The arcs $I_i$ are random, but Lemma \ref{lem:vc-arcs} controls all arcs simultaneously.  Thus
\[
  \max_i\left|\frac{d_i}{n}-A_i\right|
  \leq \sup_{I\text{ arc}} |\lambda_n(I)-\lambda(I)|\to0
\]
almost surely.
\end{proof}

\begin{lemma}[Maximum spacing on the predictor circle]\label{lem:max-spacing}
Let
\[
  \Delta_n=\max_{1\leq i\leq n} [S_{N(i)}-S_i]_1
\]
be the largest spacing between adjacent empirical predictor ranks.  Then $\Delta_n\to0$ almost surely.
\end{lemma}

\begin{proof}
Fix $\varepsilon>0$ and partition $\mathbb{T}$ into $M$ half-open arcs of equal length $1/M$, with $1/M<\varepsilon/2$.  Any circular arc of length at least $\varepsilon$ contains one of these cells.  Hence, if every cell of this partition contains at least one sample point, then no empty circular arc can have length at least $\varepsilon$, and $\Delta_n<\varepsilon$.  For a fixed cell, the probability of being empty is $(1-1/M)^n$.  Therefore
\[
  \mathbb{P}\{\Delta_n\geq \varepsilon\}
  \leq M(1-1/M)^n,
\]
which is summable in $n$.  By Borel--Cantelli, $\Delta_n<\varepsilon$ eventually almost surely.  Applying this to a countable sequence $\varepsilon\downarrow0$ proves $\Delta_n\to0$ almost surely.
\end{proof}

\begin{lemma}[Successor Fourier average]\label{lem:nn-fourier}
For each fixed integer $m\geq 1$,
\begin{equation}\label{eq:nn-fourier}
  \frac1n\sum_{i=1}^n e^{2\pi i m U_{N(i)}}e^{-2\pi i m U_i}
  \longrightarrow \mathbb{E} |g_m(S)|^2
\end{equation}
almost surely.
\end{lemma}

\begin{proof}
Set
\[
  W_i=e^{2\pi i mU_i},\qquad g(s)=\mathbb{E}(W_i\mid S_i=s).
\]
Condition on $S_1,\ldots,S_n$.  Then $W_1,\ldots,W_n$ are independent, bounded by one, and
\[
  \mathbb{E}(W_i\mid S_1,\ldots,S_n)=g(S_i).
\]
The conditional expectation of the left-hand side of \eqref{eq:nn-fourier} is
\[
  B_n=\frac1n\sum_{i=1}^n g(S_{N(i)})\overline{g(S_i)}.
\]
We first prove
\begin{equation}\label{eq:Bn-limit}
  B_n\longrightarrow \int_\mathbb{T} |g(s)|^2\,ds
\end{equation}
almost surely.  If $g$ is continuous, then by Lemma \ref{lem:max-spacing},
\[
  \max_i |g(S_{N(i)})-g(S_i)|\to0
\]
almost surely.  Therefore
\[
  \left|B_n-\frac1n\sum_{i=1}^n |g(S_i)|^2\right|
  \leq \max_i |g(S_{N(i)})-g(S_i)|\,\frac1n\sum_{i=1}^n |g(S_i)|\to0.
\]
The empirical average of $|g(S_i)|^2$ converges almost surely to $\int |g|^2$ by the strong law.

For general $g\in L^2(\mathbb{T})$, fix a continuous $g_0$ and write $r=g-g_0$.  Since $N$ is a permutation of the sample indices,
\[
  \frac1n\sum_i |r(S_{N(i)})|^2=\frac1n\sum_i |r(S_i)|^2\to \|r\|_2^2
\]
almost surely.  Cauchy-Schwarz gives
\[
\begin{aligned}
 |B_n(g)-B_n(g_0)|
 &\leq \left(\frac1n\sum_i |r(S_{N(i)})|^2\right)^{1/2}
       \left(\frac1n\sum_i |g(S_i)|^2\right)^{1/2} \\
 &\quad +\left(\frac1n\sum_i |g_0(S_{N(i)})|^2\right)^{1/2}
       \left(\frac1n\sum_i |r(S_i)|^2\right)^{1/2}.
\end{aligned}
\]
Taking limits and then choosing $g_0$ with $\|g-g_0\|_2$ arbitrarily small proves \eqref{eq:Bn-limit}.

It remains to remove the conditional expectation.  Let
\[
  T_n=\frac1n\sum_{i=1}^n W_{N(i)}\overline{W_i}.
\]
Conditional on $S_1,\ldots,S_n$, changing one variable $W_j$ can affect only the two terms corresponding to the edges entering and leaving $j$ in the empirical cycle.  Each affected term changes by at most $2/n$, so the bounded-difference constant for each of the real-valued functions $\operatorname{Re}T_n$ and $\operatorname{Im}T_n$ is at most $4/n$.  Applying McDiarmid's inequality separately to real and imaginary parts, followed by a union bound, gives, conditionally on the $S_i$'s,
\[
  \mathbb{P}\{|T_n-B_n|>\varepsilon\mid S_1,\ldots,S_n\}
  \leq 4\exp\{-n\varepsilon^2/16\}.
\]
The right side is summable in $n$, and Borel-Cantelli gives $T_n-B_n\to0$ almost surely.  Combining this with \eqref{eq:Bn-limit} proves the lemma, because $\int |g|^2=\mathbb{E}|g_m(S)|^2$.
\end{proof}

\begin{theorem}[Strong consistency]\label{thm:consistency}
If $X$ and $Y$ have non-atomic circular marginal distributions, then
\[
  \xi_n^\circ(X\to Y)\longrightarrow \xi^\circ(X\to Y)
\]
almost surely.  The corrected statistic $\xi_{n,*}^\circ$ has the same almost sure limit.
\end{theorem}

\begin{proof}
The function $h(t)=t(1-t)$ is Lipschitz on $[0,1]$.  By Lemma \ref{lem:rank-approx},
\begin{equation}\label{eq:h-rank-replace}
  \frac1n\sum_{i=1}^n h(d_i/n)-\frac1n\sum_{i=1}^n h(A_i)\longrightarrow0
\end{equation}
almost surely.

The Fourier expansion \eqref{eq:h-fourier} gives, uniformly and absolutely,
\[
  h([v-u]_1)=\frac16-\frac1{\pi^2}\sum_{m=1}^\infty
  \frac{\Re\{e^{2\pi i m v}e^{-2\pi i m u}\}}{m^2}.
\]
Thus
\[
  \frac1n\sum_{i=1}^n h(A_i)
  =\frac16-\frac1{\pi^2}\sum_{m=1}^\infty\frac1{m^2}
  \Re\left\{\frac1n\sum_{i=1}^n e^{2\pi i mU_{N(i)}}e^{-2\pi i mU_i}\right\}.
\]
For each fixed $m$, Lemma \ref{lem:nn-fourier} identifies the limit of the term in braces.  Since the terms are bounded by one and $\sum m^{-2}<\infty$, dominated convergence for series gives
\begin{equation}\label{eq:h-limit}
  \frac1n\sum_{i=1}^n h(A_i)
  \longrightarrow
  \frac16-\frac1{\pi^2}\sum_{m=1}^\infty\frac{\mathbb{E}|g_m(S)|^2}{m^2}
\end{equation}
almost surely.

Finally,
\[
  \xi_n^\circ
  =1-\frac{6}{n^2(n+1)}\sum_i d_i(n-d_i)
  =1-\frac{6n}{n+1}\left\{\frac1n\sum_i h(d_i/n)\right\}.
\]
Combining this identity with \eqref{eq:h-rank-replace} and \eqref{eq:h-limit} gives exactly the Fourier population value \eqref{eq:fourier-pop}.  Since $a_n\to1$, the corrected statistic has the same almost sure limit.
\end{proof}

\section{Null distribution under independence}\label{sec:null}

Assume in this section that $X$ and $Y$ are independent and have continuous marginals.  Then the relative cyclic order of the $Y$ ranks around the cyclic order of the $X$ ranks is a uniformly random cyclic order.  Consequently the finite-sample null distribution of $\xi_n^\circ$ is distribution-free.

Let $\pi=(\pi_1,\ldots,\pi_n)$ be a uniformly random permutation of $\{0,1,\ldots,n-1\}$, viewed cyclically with $\pi_{n+1}=\pi_1$.  Define
\[
  D_k=(\pi_{k+1}-\pi_k)\bmod n,
  \qquad
  Z_k=D_k(n-D_k),
  \qquad
  S_n=\sum_{k=1}^n Z_k.
\]
Then under the continuous independence null,
\begin{equation}\label{eq:null-representation}
  \xi_n^\circ=1-\frac{6}{n^2(n+1)}S_n.
\end{equation}

\subsection{Exact first two moments}

We use the following elementary sums.
\begin{align}
  A_n&:=\sum_{d=1}^{n-1}d(n-d)=\frac{n(n^2-1)}6,\label{eq:An}\\
  B_n&:=\sum_{d=1}^{n-1}d^2(n-d)^2=\frac{n(n^4-1)}{30}.\label{eq:Bn}
\end{align}
Let $w(d)=d(n-d)$.

\begin{lemma}[One-edge and two-edge moments]\label{lem:null-covariances}
For any $k$,
\begin{align}
  \mathbb{E} Z_k&=\frac{n(n+1)}6,\label{eq:Ez}\\
  \operatorname{Var}(Z_k)&=\frac{n(n-3)(n-2)(n+1)}{180}.\label{eq:Varz}
\end{align}
If two distinct cyclic edges are adjacent, then
\begin{equation}\label{eq:cov-adj-new}
  \operatorname{Cov}(Z_j,Z_k)=-\frac{n(n-3)(n+1)}{180}.
\end{equation}
If two distinct cyclic edges are disjoint, then
\begin{equation}\label{eq:cov-dis-new}
  \operatorname{Cov}(Z_j,Z_k)=\frac{n(n+1)}{90}.
\end{equation}
\end{lemma}

\begin{proof}
For a single edge, $D_k$ is uniform on $\{1,\ldots,n-1\}$, so \eqref{eq:Ez} follows from \eqref{eq:An}.  Also
\[
  \mathbb{E} Z_k^2=\frac{B_n}{n-1}=\frac{n(n+1)(n^2+1)}{30},
\]
and subtracting $\{n(n+1)/6\}^2$ gives \eqref{eq:Varz}.

For adjacent edges, condition on the first increment being $d$.  The second increment may be any element of $\{1,\ldots,n-1\}$ except $n-d$, because that excluded value would return to the first vertex.  Since $w(n-d)=w(d)$,
\[
  \mathbb{E}(Z_1Z_2)
  =\frac1{n-1}\sum_{d=1}^{n-1}w(d)\,\frac{A_n-w(d)}{n-2}
  =\frac{A_n^2-B_n}{(n-1)(n-2)}.
\]
Subtracting $\{n(n+1)/6\}^2$ gives \eqref{eq:cov-adj-new}.

For disjoint edges, fix the first oriented edge to be $(0,r)$; by rotation invariance this loses no generality.  For a proposed second increment $e$, the number of possible starting vertices $c$ such that $(c,c+e)$ uses neither $0$ nor $r$ is
\[
  n-4+\mathbf{1}\{e=r\}+\mathbf{1}\{e=n-r\}.
\]
Indeed, among the $n$ possible starts, the forbidden set is
\[
  \{0,r,-e,r-e\},
\]
whose cardinality is $4-\mathbf{1}\{e=r\}-\mathbf{1}\{e=n-r\}$.  Therefore, conditional on the first increment $r$, the average weight of the second disjoint edge is
\[
  \frac{(n-4)A_n+2w(r)}{(n-2)(n-3)}.
\]
Averaging over $r$ gives
\[
  \mathbb{E}(Z_1Z_3)
  =\frac{(n-4)A_n^2+2B_n}{(n-1)(n-2)(n-3)}.
\]
Subtracting $\{n(n+1)/6\}^2$ gives \eqref{eq:cov-dis-new}.
\end{proof}

\begin{theorem}[Exact null mean and variance]\label{thm:null-moments}
Under the continuous independence null,
\begin{align*}
  \mathbb{E}_0\xi_n^\circ&=0,\\
  \operatorname{Var}_0(\xi_n^\circ)&=\frac{(n-3)(n-2)}{5n^2(n+1)}.
\end{align*}
For the corrected statistic \eqref{eq:corrected-stat}, $n\geq4$,
\begin{align*}
  \mathbb{E}_0\xi_{n,*}^\circ&=0,\\
  \operatorname{Var}_0(\xi_{n,*}^\circ)&=\frac{n+1}{5(n-2)(n-3)}.
\end{align*}
\end{theorem}

\begin{proof}
By \eqref{eq:Ez},
\[
  \mathbb{E} S_n=n\frac{n(n+1)}6=\frac{n^2(n+1)}6.
\]
The representation \eqref{eq:null-representation} therefore gives $\mathbb{E}_0\xi_n^\circ=0$.

There are $n$ variances, $n$ unordered adjacent edge pairs, and $n(n-3)/2$ unordered disjoint edge pairs.  Lemma \ref{lem:null-covariances} gives
\begin{align*}
  \operatorname{Var}(S_n)
  &=n\operatorname{Var}(Z_1)+2n\operatorname{Cov}_{\rm adj}+n(n-3)\operatorname{Cov}_{\rm dis}\\
  &=\frac{n^2(n+1)(n-2)(n-3)}{180}.
\end{align*}
Multiplying by $\{6/[n^2(n+1)]\}^2$ yields the variance of $\xi_n^\circ$.  Division by $a_n^2$, with $a_n=(n-2)(n-3)/\{n(n+1)\}$, gives the corrected variance.
\end{proof}

\subsection{Central limit theorem}

The null CLT follows from a combinatorial CLT for one-cycle permutation sums.  The reduction is as follows.  A cyclic ordering $\pi_1,\ldots,\pi_n$ defines a successor permutation $\sigma$ on $\{0,\ldots,n-1\}$ by
\[
  \sigma(\pi_k)=\pi_{k+1},\qquad k=1,\ldots,n,
\]
where $\pi_{n+1}=\pi_1$.  If $\pi$ is a uniformly random cyclic order, then $\sigma$ is uniformly distributed over the set of all $n$-cycles.  Conversely, every directed $n$-cycle has exactly $n$ cyclic listings, so this representation is exact.

Define the symmetric circulant weight matrix
\[
  W_{ij}= [j-i]_n\{n-[j-i]_n\},\qquad W_{ii}=0,
\]
where $[j-i]_n$ is the representative in $\{0,1,\ldots,n-1\}$.  Let
\[
  \bar W=\frac1{n-1}\sum_{j\ne i}W_{ij}=\frac{n(n+1)}6,
\]
which is independent of $i$, and set
\[
  a_{ij}=\begin{cases}
    W_{ij}-\bar W,& i\ne j,\\
    0,& i=j.
  \end{cases}
\]
Then $a_{ij}=a_{ji}$, $a_{ii}=0$, and every row sum is zero.  Moreover,
\begin{equation}\label{eq:linear-cycle-stat}
  S_n-\mathbb{E} S_n=\sum_{i=0}^{n-1} a_{i,\sigma(i)}.
\end{equation}

We use the following theorem of Goldstein \cite[Theorem 2.6]{goldstein2005}.  It is a Berry-Esseen bound for combinatorial sums over random permutations whose distribution is constant on a fixed cycle type and has no fixed points; the special case needed here is the uniform distribution over $n$-cycles.

\begin{theorem}[Goldstein's combinatorial CLT]\label{thm:goldstein}
Let $\sigma$ be uniformly distributed over the $n$-cycles of $\{1,\ldots,n\}$.  Let $a_{ij}=a_{ji}$, $a_{ii}=0$, and $\sum_j a_{ij}=0$ for every $i$.  Put
\[
  Y_n=\sum_i a_{i,\sigma(i)},\qquad \sigma_n^2=\operatorname{Var}(Y_n),
  \qquad C_n=\max_{i,j}|a_{ij}|.
\]
If $C_n/\sigma_n\to0$, then
\[
  \frac{Y_n}{\sigma_n}\Rightarrow N(0,1).
\]
\end{theorem}

\begin{theorem}[Asymptotic null law]\label{thm:null-clt}
Under the continuous independence null,
\begin{equation}\label{eq:null-clt}
  \sqrt n\,\xi_n^\circ\Rightarrow N(0,1/5).
\end{equation}
The same limit holds for $\sqrt n\,\xi_{n,*}^\circ$.
\end{theorem}

\begin{proof}
By \eqref{eq:linear-cycle-stat}, the centered null statistic is a one-cycle combinatorial sum.  The conditions of Theorem \ref{thm:goldstein} have just been verified, since the matrix is symmetric, has zero diagonal, and has zero row sums.

It remains only to check $C_n/\sigma_n\to0$.  Since $0\leq W_{ij}\leq n^2/4$ and $\bar W=n(n+1)/6$,
\[
  C_n\leq \frac{n^2}{4}+\frac{n(n+1)}6=O(n^2).
\]
On the other hand, Theorem \ref{thm:null-moments} gives
\[
  \sigma_n^2=\operatorname{Var}(S_n)=\frac{n^2(n+1)(n-2)(n-3)}{180}=O(n^5),
\]
so $C_n/\sigma_n=O(n^{-1/2})\to0$.  Therefore
\[
  \frac{S_n-\mathbb{E} S_n}{\sqrt{\operatorname{Var}(S_n)}}\Rightarrow N(0,1).
\]
Using \eqref{eq:null-representation} and the exact variance in Theorem \ref{thm:null-moments},
\[
  \operatorname{Var}(\sqrt n\,\xi_n^\circ)
  = n\frac{(n-3)(n-2)}{5n^2(n+1)}\longrightarrow \frac15,
\]
which proves \eqref{eq:null-clt}.  Since $a_n\to1$, the corrected statistic has the same limiting distribution.
\end{proof}

\begin{remark}[What is distribution-free?]
The finite-sample null distribution is free of the marginal distributions whenever both marginals are continuous.  It depends only on a uniform random cyclic permutation, or equivalently on a uniformly random $n$-cycle.  If ties are present, the exact finite-sample distribution depends on the tie-breaking convention, as in the ordinary Chatterjee coefficient.
\end{remark}

\section{Examples}

The Fourier formula gives closed-form expressions in common circular models.

\begin{example}[Circular additive noise]\label{ex:additive}
Let
\[
  U=S+\varepsilon\pmod 1,
\]
where $S$ is uniform on $\mathbb{T}$ and independent of the circular noise $\varepsilon$. Then
\[
  g_m(S)=e^{2\pi i m S}\,\mathbb{E} e^{2\pi i m\varepsilon},
\]
so
\begin{equation}\label{eq:additive-noise}
  \xi^\circ(X\to Y)=\frac{6}{\pi^2}\sum_{m=1}^{\infty}
  \frac{|\varphi_\varepsilon(m)|^2}{m^2},
  \qquad
  \varphi_\varepsilon(m):=\mathbb{E} e^{2\pi i m\varepsilon}.
\end{equation}
No noise gives $\varphi_\varepsilon(m)=1$ and $\xi^\circ=1$. Uniform noise gives $\varphi_\varepsilon(m)=0$ for all $m\geq1$ and $\xi^\circ=0$.
\end{example}

\begin{example}[Wrapped normal noise]
If $\varepsilon$ is wrapped normal with unwrapped variance $\sigma^2$, then
\[
  \varphi_\varepsilon(m)=\exp(-2\pi^2\sigma^2m^2),
\]
and therefore
\[
  \xi^\circ=\frac{6}{\pi^2}\sum_{m=1}^{\infty}
  \frac{\exp(-4\pi^2\sigma^2m^2)}{m^2}.
\]
\end{example}

\begin{example}[Von Mises noise]
If $\varepsilon$ has a von Mises distribution with concentration $\kappa$ and mean direction zero, then
\[
  \varphi_\varepsilon(m)=\frac{I_m(\kappa)}{I_0(\kappa)},
\]
where $I_m$ is the modified Bessel function. Hence
\[
  \xi^\circ=\frac{6}{\pi^2}\sum_{m=1}^{\infty}
  \frac{1}{m^2}\left\{\frac{I_m(\kappa)}{I_0(\kappa)}\right\}^2.
\]
This increases from $0$ at $\kappa=0$ to $1$ as $\kappa\to\infty$.
\end{example}

\begin{example}[Uniform arc noise]
If $\varepsilon$ is uniform on an arc of length $a\in(0,1]$, then
\[
  |\varphi_\varepsilon(m)|=\left|\frac{\sin(\pi m a)}{\pi m a}\right|,
\]
and
\[
  \xi^\circ=\frac{6}{\pi^2}\sum_{m=1}^{\infty}
  \frac{1}{m^2}\left\{\frac{\sin(\pi m a)}{\pi m a}\right\}^2.
\]
At $a=1$ the noise is uniform on the circle and the coefficient is zero.
\end{example}

\section{Algorithm}

For no-tie data the computation is $O(n\log n)$.

\begin{enumerate}[leftmargin=2em]
\item Convert angles to representatives in $[0,1)$.
\item Sort the indices by $X$ to obtain the cyclic order $i_1,\ldots,i_n$.
\item Sort the indices by $Y$ to obtain cyclic ranks $\rho_j\in\{0,\ldots,n-1\}$.
\item For $k=1,\ldots,n$, compute $d_k=(\rho_{i_{k+1}}-\rho_{i_k})\bmod n$, with $i_{n+1}=i_1$.
\item Return \eqref{eq:raw-stat}, or return the finite-sample corrected version \eqref{eq:corrected-stat} when $n\geq4$.
\end{enumerate}

For inference under continuous independence, use the asymptotic normal approximation
\[
  \sqrt{5n}\,\xi_n^\circ\approx N(0,1)
\]
for the raw statistic, or the corresponding finite-sample variance in Theorem \ref{thm:null-moments}. A permutation test is also exact conditionally on the observed circular ranks. To carry it out, permute the $Y$ ranks relative to the $X$ cyclic order and recompute the statistic.

\subsection{Symmetrization}

The coefficient is directed. It measures the extent to which the circular response is a measurable function of the circular predictor. The symmetric circular Chatterjee coefficient is
\[
  \xi_{n,\mathrm{sym}}^\circ(X,Y)=\max\{\xi_n^\circ(X\to Y),\xi_n^\circ(Y\to X)\},
\]
and similarly at the population level. By Theorem \ref{thm:fourier-endpoints}, the population symmetric version is zero if and only if $X$ and $Y$ are independent, and one if and only if either variable is a measurable circular function of the other.

Several extensions are immediate. For a circular response and a Euclidean or manifold-valued predictor, replace the cyclic successor graph in $X$ by a local proximity graph and keep the circular response discrepancy $H$. This links the present construction to graph-based Chatterjee-type measures on general spaces. The special feature of the circle-circle case is that the predictor local graph is the cyclic order, and the response discrepancy can be computed entirely from cyclic ranks.

\section{Simulation study}\label{sec:simulation}

We now compare the finite-sample behavior of the proposed cyclic-rank coefficient with several natural competitors.  The purpose of the experiment is to compare the numerical values of the coefficients, not to study level-calibrated tests.  Thus the reported results are coefficient summaries rather than rejection probabilities.

All simulations used $n=200$ observations and $1000$ Monte Carlo replications.  Angles are represented in radians modulo $2\pi$.  In every non-null model, $X\sim \mathrm{Unif}(0,2\pi)$ and the noise variable $\varepsilon_\sigma$ is a wrapped normal obtained by taking $N(0,\sigma^2)$ modulo $2\pi$.  We considered the following models.
\begin{align*}
\text{independence}\quad &Y\sim \mathrm{Unif}(0,2\pi),\quad Y\perp X,\\
\text{rotation}\quad &Y=X+\pi/4+\varepsilon_\sigma,\\
\text{doubling}\quad &Y=2X+\varepsilon_\sigma,\\
\text{quadrupling}\quad &Y=4X+\varepsilon_\sigma,\\
\text{antipodal mixture}\quad &Y=X+Z+\varepsilon_\sigma,\quad Z\in\{0,\pi\}\text{ with equal probability},\\
\text{localized bump}\quad &Y=X+1.25\exp\{2\cos(X-\pi)\}/\exp(2)+\varepsilon_\sigma.
\end{align*}
We also simulated a two-level step model, $Y=\pi/4$ for $X<\pi$ and $Y=5\pi/4$ for $X\geq \pi$, plus noise.  This model is useful for exposing the effect of ties, but because its zero-noise response has atoms, it should not be treated as a clean no-tie benchmark for the coefficient developed in the main theory.

The competitors were as follows.  The first was Chatterjee's standard-Borel cut construction, implemented by cutting both circles at zero and applying the ordinary real-line coefficient.  We also computed a cut-averaged Borel version by averaging over an $8\times8$ grid of predictor and response cut points.  The second competitor was a Deb-Ghosal-Sen style kernel coefficient, specialized to circular data by using a circular graph joining each predictor point to its five closest predictor observations and the von Mises kernel $K(y,y')=\exp\{2\cos(y-y')\}$ in the response.  Finally, we computed the absolute Jammalamadaka-Sengupta and Fisher-Lee circular correlations, denoted JS and FL in the tables.

\begin{table}[H]
\centering
\small
\caption{Mean coefficients over 1000 Monte Carlo replications with $n=200$. The Borel average uses an $8\times 8$ grid of predictor and response cut points. The Deb-Ghosal-Sen style coefficient uses a circular graph joining each predictor point to its five closest predictor observations and a von Mises kernel with $\lambda=2$.}
\label{tab:simulation-means}
\begin{tabular}{llrrrrrr}
\toprule
Model & $\sigma$ & $\xi_n^\circ$ & Borel $0$ & Borel avg. & Kernel & JS & FL \\
\midrule
Independence & 0.000 & -0.001 & -0.002 & -0.001 & 0.000 & 0.055 & 0.005 \\
Rotation & 0.000 & 0.970 & 0.970 & 0.972 & 0.995 & 1.000 & 1.000 \\
Rotation & 0.500 & 0.526 & 0.530 & 0.532 & 0.591 & 0.744 & 0.778 \\
Doubling & 0.000 & 0.941 & 0.956 & 0.943 & 0.981 & 0.056 & 0.005 \\
Doubling & 0.500 & 0.524 & 0.535 & 0.529 & 0.587 & 0.059 & 0.005 \\
Quadrupling & 0.000 & 0.885 & 0.899 & 0.886 & 0.927 & 0.058 & 0.005 \\
Quadrupling & 0.500 & 0.517 & 0.527 & 0.522 & 0.569 & 0.058 & 0.005 \\
Antipodal mixture & 0.000 & 0.236 & 0.240 & 0.237 & 0.289 & 0.056 & 0.006 \\
Antipodal mixture & 0.500 & 0.056 & 0.056 & 0.057 & 0.100 & 0.055 & 0.006 \\
Localized bump & 0.000 & 0.968 & 0.976 & 0.970 & 0.993 & 0.620 & 0.846 \\
Localized bump & 0.500 & 0.479 & 0.499 & 0.502 & 0.575 & 0.529 & 0.648 \\
\bottomrule
\end{tabular}
\end{table}

\begin{figure}[H]
\centering
\includegraphics[width=\textwidth]{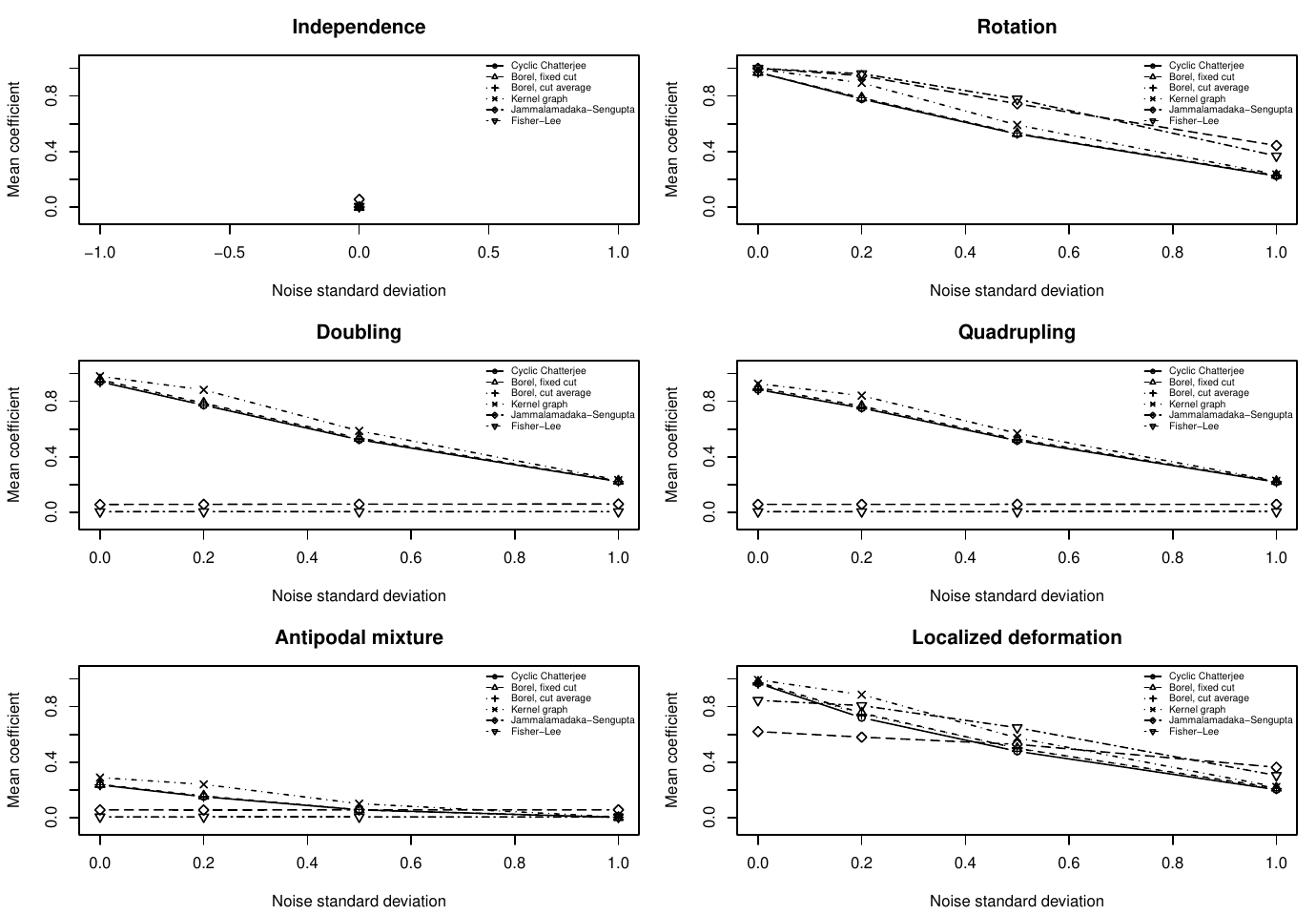}
\caption{Mean coefficient values as a function of the wrapped-normal noise standard deviation $\sigma$.  The multi-winding models show the main advantage of the proposed cyclic-rank construction. The models $Y=2X$ and $Y=4X$ are deterministic circular functions of $X$ at zero noise, but standard first-order circular correlations are close to zero.}
\label{fig:simulation-curves}
\end{figure}

Table \ref{tab:simulation-means} and Figure \ref{fig:simulation-curves} show three main patterns.  First, under independence the proposed statistic, the fixed-cut Borel version, the cut-averaged Borel version, and the kernel graph coefficient are all centered near zero.  The absolute JS coefficient is positive under independence because the absolute value is reported, but it is still small.

Second, for the simple rotation model, the standard circular correlations perform very well.  This is expected because rotation is precisely the type of first-order circular association that these coefficients are designed to capture.  The proposed statistic remains large and decreases smoothly as noise increases, but it is not meant to dominate first-order circular correlations in this case.

Third, the multi-winding models expose the limitation of standard circular correlations.  At zero noise the doubling and quadrupling models are deterministic circular functions of $X$.  The proposed coefficient has mean values $0.941$ and $0.885$, respectively, while JS is approximately $0.056$-$0.058$ and FL is approximately $0.005$.  At noise level $\sigma=0.5$, the proposed statistic remains around $0.52$ in both models, whereas JS and FL remain essentially at their null levels.  The kernel graph coefficient is also strong in these examples, but it depends on the choices of the graph parameter $k$ and the response kernel bandwidth.  The cyclic-rank coefficient is tuning-free.

The antipodal mixture illustrates a non-functional but still non-independent circular relationship.  Conditional on $X$, the response has two antipodal possibilities.  The proposed coefficient and the kernel graph coefficient detect this moderate dependence at small noise, while JS and FL are essentially blind to it.  The localized-bump model is favorable to several methods because it has a strong smooth first-order component; all reasonable coefficients decrease with noise in this case.

\begin{table}[H]
\centering
\small
\caption{Sensitivity of the Borel cut-point construction to the choice of cut points. For each Monte Carlo data set, the ordinary Chatterjee coefficient was evaluated over an $8\times 8$ grid of cut points; the table reports Monte Carlo averages of the grid mean, grid standard deviation, grid minimum, and grid maximum.}
\label{tab:borel-sensitivity}
\begin{tabular}{llrrrr}
\toprule
Model & $\sigma$ & Grid mean & Grid SD & Grid min & Grid max \\
\midrule
Independence & 0.000 & -0.001 & 0.030 & -0.053 & 0.050 \\
Doubling & 0.500 & 0.529 & 0.043 & 0.458 & 0.600 \\
Quadrupling & 0.500 & 0.522 & 0.042 & 0.453 & 0.590 \\
Antipodal mixture & 0.000 & 0.237 & 0.028 & 0.193 & 0.281 \\
Localized bump & 0.200 & 0.747 & 0.065 & 0.600 & 0.820 \\
Step arc (ties) & 0.200 & 0.395 & 0.160 & 0.096 & 0.494 \\
\bottomrule
\end{tabular}
\end{table}

\begin{figure}[H]
\centering
\includegraphics[width=0.75\textwidth]{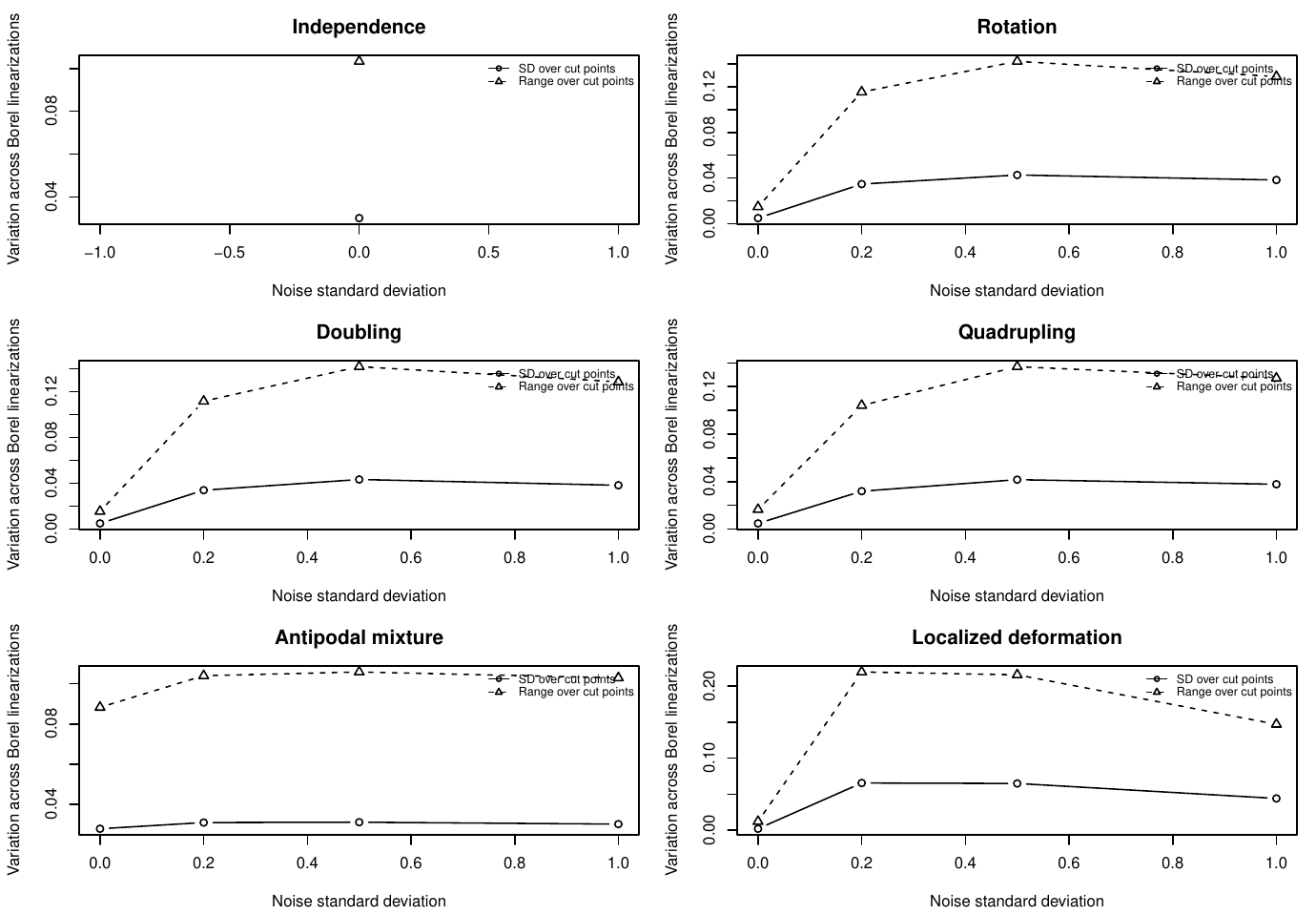}
\caption{Sensitivity of the cut-based Borel construction to the choice of cut point, measured by the standard deviation across an $8\times8$ grid of cut points.  The variability is small in some noiseless smooth cases but can be appreciable for noisy or discontinuous circular relationships.}
\label{fig:borel-cut-sd}
\end{figure}

Table \ref{tab:borel-sensitivity} and Figure \ref{fig:borel-cut-sd} quantify the cut-point dependence.  In many smooth continuous examples, the fixed-cut and cut-averaged Borel coefficients have similar means.  However, the cut-based value is not intrinsic.  For instance, in the localized-bump model with $\sigma=0.2$, the average minimum and maximum over the cut-point grid are $0.600$ and $0.820$.  In the step-arc model with $\sigma=0.2$, they are $0.096$ and $0.494$.  These examples illustrate why the standard-Borel construction is mathematically valid but statistically non-canonical on the circle.  The cyclic-rank coefficient removes the cut-point choice altogether.

The step-arc example also highlights the role of atoms.  At zero noise the response has only two values, and the simulation code breaks ties by tiny random jitter.  This is why the proposed no-tie statistic does not approach one in that example, even though the response is formally a measurable function of the predictor.  A tie-aware version, based either on randomized circular ranks or on an exact block-rank convention, is needed for discrete circular responses.  This is the same phenomenon that appears for the ordinary Chatterjee coefficient in the presence of ties.

\subsection{Null calibration and power}

The previous experiments compared the numerical values of several coefficients. We now examine the use of the proposed statistic as a test statistic for independence. Under the continuous independence null, Theorem~\ref{thm:null-moments} gives
\[
  \operatorname{Var}_0(\xi_n^\circ)
  =
  \frac{(n-3)(n-2)}{5n^2(n+1)}.
\]
Thus a one-sided normal-approximation test rejects for large values of
\[
  Z_n
  =
  \frac{\xi_n^\circ}
       {\{(n-3)(n-2)/(5n^2(n+1))\}^{1/2}}.
\]
We compared this test with the exact conditional permutation test obtained by fixing the cyclic order of the predictor ranks and randomly permuting the response cyclic ranks. The permutation test used 499 random permutations, so the smallest attainable permutation \(p\)-value was \(1/500=0.002\). All rejection probabilities below are reported at nominal level \(0.05\), based on 1000 Monte Carlo replications.

Table~\ref{tab:null-size-normal-perm} reports empirical size under independence for \(n=30,50,100,200\). Both tests are close to the nominal level even for \(n=30\). The normal approximation is slightly liberal for \(n=30\) and \(n=50\), and slightly conservative for \(n=100\) and \(n=200\), but the deviations are small. The permutation test behaves as expected and stays essentially at the nominal level.

\begin{table}[H]
\centering
\small
\caption{Empirical size under independence at nominal level \(0.05\). The normal test uses the exact finite-sample null variance from Theorem~\ref{thm:null-moments}. The permutation test permutes the response cyclic ranks relative to the predictor cyclic order.}
\label{tab:null-size-normal-perm}
\begin{tabular}{rrrrr}
\toprule
\(n\) & Mean \(\xi_n^\circ\) & SD \(\xi_n^\circ\) & Normal test & Permutation test \\
\midrule
30  &  0.000 & 0.073 & 0.053 & 0.050 \\
50  &  0.000 & 0.060 & 0.053 & 0.049 \\
100 & -0.001 & 0.044 & 0.046 & 0.044 \\
200 &  0.000 & 0.031 & 0.044 & 0.045 \\
\bottomrule
\end{tabular}
\end{table}

\begin{figure}[H]
\centering
\includegraphics[width=0.65\textwidth]{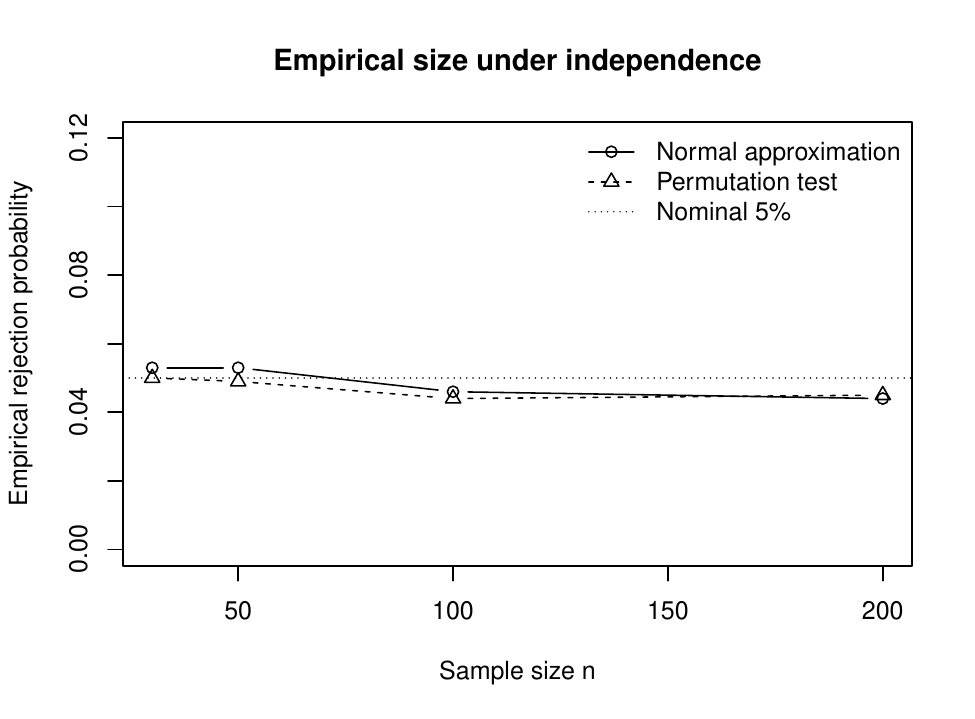}
\caption{Empirical rejection probability under independence for the normal approximation and the conditional permutation test. The horizontal line marks the nominal level \(0.05\).}
\label{fig:null-size-normal-perm}
\end{figure}

Table~\ref{tab:power-normal-perm-n200} and Figure~\ref{fig:power-normal-perm-n200} report power for \(n=200\). The alternatives are the rotation, doubling, quadrupling, and antipodal-mixture models described above. The normal approximation and the permutation test give nearly identical rejection probabilities across all alternatives. For the deterministic and noisy functional relationships, the power is essentially one even at the largest noise level considered. This includes the multi-winding alternatives \(Y=2X+\varepsilon\) and \(Y=4X+\varepsilon\), where ordinary first-order circular correlations are close to their null values. The antipodal-mixture model is harder because the response is not a function of the predictor; nevertheless the test has high power for small noise and moderate power at \(\sigma=0.5\), before dropping to the null level at \(\sigma=1\).

\begin{table}[H]
\centering
\small
\caption{Empirical power for \(n=200\) at nominal level \(0.05\).}
\label{tab:power-normal-perm-n200}
\begin{tabular}{llrrr}
\toprule
Model & \(\sigma\) & Mean \(\xi_n^\circ\) & Normal test & Permutation test \\
\midrule
Rotation & 0.0 & 0.970 & 1.000 & 1.000 \\
Rotation & 0.2 & 0.779 & 1.000 & 1.000 \\
Rotation & 0.5 & 0.526 & 1.000 & 1.000 \\
Rotation & 1.0 & 0.220 & 1.000 & 1.000 \\
\midrule
Doubling & 0.0 & 0.941 & 1.000 & 1.000 \\
Doubling & 0.2 & 0.773 & 1.000 & 1.000 \\
Doubling & 0.5 & 0.524 & 1.000 & 1.000 \\
Doubling & 1.0 & 0.221 & 1.000 & 1.000 \\
\midrule
Quadrupling & 0.0 & 0.885 & 1.000 & 1.000 \\
Quadrupling & 0.2 & 0.751 & 1.000 & 1.000 \\
Quadrupling & 0.5 & 0.516 & 1.000 & 1.000 \\
Quadrupling & 1.0 & 0.220 & 1.000 & 1.000 \\
\midrule
Antipodal mixture & 0.0 & 0.235 & 1.000 & 1.000 \\
Antipodal mixture & 0.2 & 0.147 & 0.979 & 0.982 \\
Antipodal mixture & 0.5 & 0.055 & 0.552 & 0.540 \\
Antipodal mixture & 1.0 & 0.003 & 0.055 & 0.055 \\
\bottomrule
\end{tabular}
\end{table}

\begin{figure}[H]
\centering
\includegraphics[width=\textwidth]{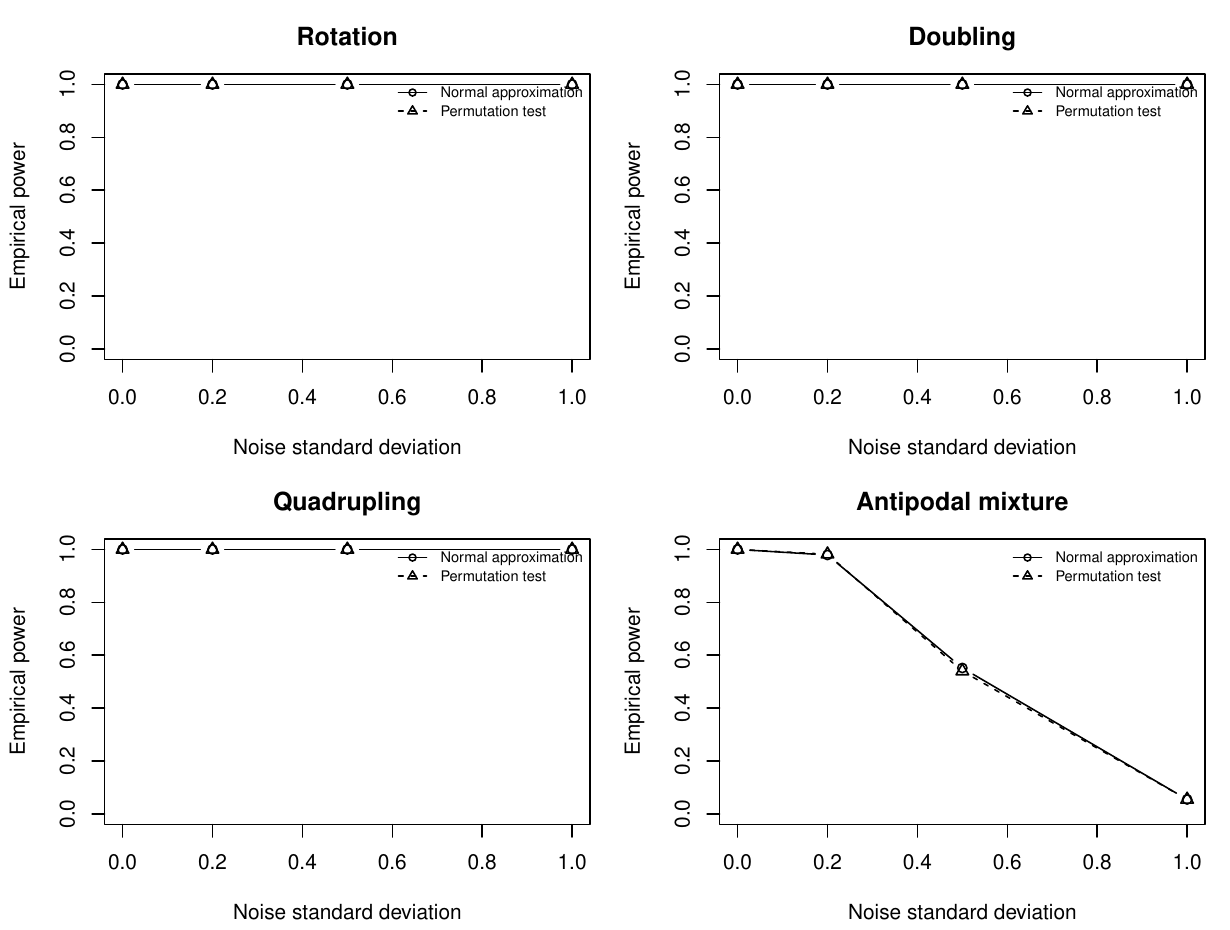}
\caption{Empirical power for the proposed cyclic Chatterjee statistic at \(n=200\). The normal approximation uses the exact finite-sample null variance, while the permutation test permutes the response cyclic ranks. The two curves are almost indistinguishable in these experiments.}
\label{fig:power-normal-perm-n200}
\end{figure}

These results support two conclusions. First, the asymptotic normal approximation, when combined with the exact finite-sample null variance, gives a reliable and computationally inexpensive test for moderate sample sizes. Second, the proposed statistic has strong power for functional circular relationships, including multi-winding relationships that are difficult for standard circular correlations. The antipodal-mixture example illustrates the expected limitation. When the relationship is non-functional and heavily noise-contaminated, the coefficient decreases toward its null behavior.

\section{Discussion}

This paper develops a circular analogue of Chatterjee's coefficient that is intrinsic to cyclic order. Starting from the observation that the standard-Borel extension can be applied to the circle only after choosing cut points, we averaged the ordinary coefficient over response cuts in circular rank space and, in finite samples, over predictor and response sample cut gaps. The finite-sample average reduces to a simple cyclic-rank statistic, with penalty $d(n-d)$ for the circular response-rank increment across each adjacent predictor edge. Thus no origin, artificial boundary point, or tuning parameter is introduced.

At the population level, the same construction leads to a coefficient $\xi^\circ(X\to Y)$ with three complementary descriptions, namely conditional circular dispersion, cut-average, and Fourier series. The Fourier representation gives the zero-one theory. For non-atomic circular marginals, the coefficient lies in $[0,1]$, is zero exactly under independence, and is one exactly when the circular response is a measurable function of the circular predictor. We also proved strong consistency, derived an exact distribution-free null mean and variance under continuous independence, and obtained an asymptotic normal null law. The simulations illustrate the intended use of the statistic. It remains sensitive to functional circular relationships, including multi-winding relationships that first-order circular correlations may miss, while preserving a simple permutation-based calibration.

A natural next step is to develop a tie-aware version for atomic or mixed circular responses. The present theory assumes non-atomic marginals, and the simulations show that ties can matter in step-type circular relationships. A satisfactory extension would combine randomized or block circular ranks with an intrinsic cyclic penalty, preserve the zero and one interpretations, and give a finite-sample null law when ties are present.

\section*{Statement on AI usage}
The author used ChatGPT to generate the R code for the simulation study, help draft portions of the manuscript, and assist in verifying the mathematical results.

\bibliographystyle{plain}
\bibliography{Paper}

\end{document}